\documentclass[titlepage,12pt]{article} 
\usepackage{hyperref}
\usepackage{amssymb,amsthm,amsmath} 
\usepackage[a4paper]{geometry}
\usepackage{datetime2}
\usepackage[utf8]{inputenc}
\usepackage[italian,english]{babel}
\usepackage{mathrsfs}

\date{}

\newcommand{\re}{\mathbb{R}}
\newcommand{\ep}{\varepsilon}

\newcommand{\Fom}{\mathscr{F}_\omega}
\newcommand{\Leb}{\ensuremath{\mathscr{L}}}
\newcommand{\ap}{\mathrm{ap}\,\mbox{--}\,}

\newcommand{\apliminf}{\ap\liminf}
\newcommand{\aplimsup}{\ap\limsup}
\newcommand{\hatK}{\widehat{K}}
\newcommand{\reg}{\operatorname{Reg}}

\newtheorem{thm}{Theorem}[section]
\newtheorem{thmbibl}{Theorem}

\newtheorem{rmk}[thm]{Remark}
\newtheorem{prop}[thm]{Proposition}

\newtheorem{cor}[thm]{Corollary}

\newtheorem{lemma}[thm]{Lemma}
\newtheorem{newopen}{Open problem}
\newtheorem{question}[thmbibl]{Question}

\title{On the characterization of constant functions through nonlocal functionals}

\author{
Massimo Gobbino\vspace{1ex}\\ 
{\normalsize Università degli Studi di Pisa} \\
{\normalsize Dipartimento di Matematica}\\ 
{\normalsize PISA (Italy)}\\  
{\normalsize e-mail: \texttt{massimo.gobbino@unipi.it}}
\and
Nicola Picenni\vspace{1ex}\\ 
{\normalsize Scuola Normale Superiore} \\
{\normalsize PISA (Italy)}\\
{\normalsize e-mail: \texttt{nicola.picenni@sns.it}}
}

\begin{document}
\maketitle

\begin{abstract}

We address a classical open question by H.~Brezis and R.~Ignat concerning the characterization of constant functions through double integrals that involve difference quotients. 

Our first result is a counterexample to the question in its full generality. This counterexample requires the construction of a function whose difference quotients avoid a sequence of intervals with endpoints that diverge to infinity. Our second result is a positive answer to the question when restricted either to functions that are bounded and approximately differentiable almost everywhere, or to functions with bounded variation.

We also present some related open problems that are motivated by our positive and negative results.

\vspace{6ex}

\noindent{\bf Mathematics Subject Classification 2010 (MSC2010):} 
26A30, 26A45, 28A50.

		
\vspace{6ex}


\end{abstract}

 
\section{Introduction}

In this paper we consider the problem of characterizing constant functions by means of nonlocal functionals depending on double integrals of difference quotients. The starting point is the following result obtained by H.~Brezis in~\cite{2002-brezis} as a by-product of the theory developed in~\cite{2001-BouBreMir} (see also~\cite{2008-mariconda,2020-Houston-Ranjbar} for alternative proofs, and~\cite{2004-ManuMat-Pietruska,2012-JMAA-Ranjbar} for an extension to metric-measure spaces).

\begin{thmbibl}[{see~\cite[Proposition~2]{2002-brezis}}]

Let $d$ be a positive integer, let $\Omega\subseteq\re^{d}$ be a connected open set, and let $p\geq 1$ be a real number.

Then a measurable function $u:\Omega\to\re$ is (essentially) constant in $\Omega$ if and only if
\begin{equation}
\iint_{\Omega\times\Omega}\frac{|f(y)-f(x)|^{p}}{|y-x|^{p}}\cdot\frac{1}{|y-x|^{d}}\,dx\,dy<+\infty.
\nonumber
\end{equation}

\end{thmbibl}  

In the same paper, H.~Brezis suggested to extend the result by investigating more general functionals of the form
\begin{equation}
\Fom (u,\Omega):=\iint_{\Omega\times\Omega}\omega \left(\frac{|u(y)-u(x)|}{|y-x|}\right) \frac{1}{|y-x|^d} \,dx\, dy,
\label{defn:F-omega}
\end{equation}
where $\omega$ belongs to the class $\mathcal{W}$ of all continuous functions $\omega :[0,+\infty) \to [0,+\infty)$ such that $\omega (0)=0$ and 
\begin{equation}
\omega (\mu)>0
\qquad
\forall\mu>0.
\label{hp:omega>0}
\end{equation}

\paragraph{\textmd{\emph{Previous literature}}}

The functionals of the form (\ref{defn:F-omega}) have been considered by R.~Ignat in~\cite{2005-ignat}, where the following question is addressed. 

\begin{question}[{see~\cite[Problem~1]{2005-ignat}}]\label{Q:ignat}

Find a necessary and sufficient condition on $\omega\in\mathcal{W}$ so that for every positive integer $d$, every connected open set $\Omega\subseteq\re^{d}$, and every measurable function $u:\Omega\to\re$ it turns out that
\begin{equation}
\Fom(u,\Omega)<+\infty \iff u \mbox{ is essentially constant in $\Omega$}.
\label{implication}
\end{equation}

\end{question}

As observed in \cite{2005-ignat}, the class $\mathcal{W}$ is the natural minimal setting for Question~\ref{Q:ignat}, because the continuity of $\omega$ guarantees that $\Fom(u,\Omega)$ is well-defined for every measurable function $u$, assumption $\omega(0)=0$ guarantees that $\Fom(u,\Omega)$ is finite (and actually equal to~0) when $u$ is a constant function, while the positivity of $\omega(\mu)$ for positive values of $\mu$ guarantees that $\Fom(u,\Omega)=+\infty$ for every Lipschitz continuous function $u$ that is non-constant.

Several partial results were proved in \cite{2005-ignat}. Let us mention some of them. 
\begin{itemize}

\item  (Positive answer with stronger assumptions on $u$, see~\cite[Theorem~1.5]{2005-ignat}). For every $\omega\in\mathcal{W}$, implication (\ref{implication}) holds true for every $u\in W^{1,1}(\Omega)$.

\item  (Necessary integral condition on $\omega$, see~\cite[Theorem~1.2]{2005-ignat}). If implication (\ref{implication}) holds true for every measurable function $u$, then necessarily $\omega$ satisfies the integral condition
\begin{equation}
\label{hp:omega-int}
\int_1 ^{+\infty} \frac{\omega(\mu)}{\mu^2} \, d\mu =+\infty.
\end{equation}

To this end, it is enough to consider the case where $d=1$, $\Omega=(-1,1)$, and $u$ is the Heaviside function equal to~0 in $(-1,0)$ and equal to~1 in $(0,1)$. We point out that this function belongs to $BV((-1,1))$, but not to $W^{1,1}((-1,1))$.

\item  (Positive answer with growth condition on $\omega$, see~\cite[Theorem~1.3]{2005-ignat}). If $\omega\in\mathcal{W}$, and in addition $\omega$ satisfies the growth condition
\begin{equation}
\liminf_{\mu\to +\infty}\frac{\omega(\mu)}{\mu}>0,
\label{hp:omega-liminf}
\end{equation}
then implication (\ref{implication}) holds true for every measurable function $u$.

\item  (Negative answer for slower growth, see~\cite[Theorem~1.7]{2005-ignat}). If $\omega(\mu)=\mu^{\theta}$ for some $\theta\in(0,1)$, then there are examples where implication (\ref{implication}) fails, even in dimension one and for functions $u$ that are H\"older continuous and with bounded variation (but of course not in $W^{1,1}$).

\end{itemize}

The results quoted above led R.~Ignat to ask whether the integral condition (\ref{hp:omega-int}), in addition to (\ref{hp:omega>0}), is also sufficient for implication (\ref{implication}) to be true.

\begin{question}[{see~\cite[Open Question~1]{2005-ignat}}]\label{open:ignat}

Let us assume that $\omega\in\mathcal{W}$ satisfies the integral condition (\ref{hp:omega-int}). 

Determine whether implication (\ref{implication}) holds true for every positive integer $d$, every connected open set $\Omega\subseteq\re^{d}$, and every measurable function $u:\Omega\to\re$.

\end{question}

\paragraph{\textmd{\emph{Our results and new open problems}}}

In this paper we present two main results. The first one is a negative answer to Question~\ref{open:ignat} in its full generality. 

\begin{thm}[Counterexample to Question~\ref{open:ignat}]\label{thm:counterexample}

There exist a continuous function $\omega:[0,+\infty)\to[0,+\infty)$ that satisfies (\ref{hp:omega>0}) and (\ref{hp:omega-int}), and a non-constant measurable function $u:(0,1)\to\re$ such that $\Fom(u,(0,1))<+\infty$. 

In addition, one can assume that $\omega$ satisfies the growth condition
\begin{equation}
\lim_{\mu\to+\infty}\frac{\omega(\mu)}{(\log\mu)^{\theta}}=+\infty
\qquad
\forall\theta>0,
\label{th:growth-omega}
\end{equation}
and that the derivative of $u$ exists and vanishes for almost every $x\in(0,1)$.
\end{thm}

\begin{rmk}
\begin{em}

The space dimension plays no role in the construction of the counterexample. Indeed, one can check that the same argument works in the hypercube $(0,1)^{d}$ by just considering a function $u$ that depends only on a single coordinate.

\end{em}
\end{rmk}

The second result of this paper is a positive answer to Question~\ref{open:ignat} under the additional assumption that $u$ is of class $L^{\infty}$ and approximately differentiable almost everywhere in the sense of~\cite[section~3.1.2]{federer:GMT}.

\begin{thm}[Approximately differentiable bounded functions]\label{thm:AD}

Let $\omega:[0,+\infty)\to[0,+\infty)$ be a continuous function that satisfies (\ref{hp:omega>0}) and (\ref{hp:omega-int}). Let $\Omega\subseteq\re^{d}$ be a connected open set, and let $u\in L^{\infty}(\Omega)$ be a non-constant function that is approximately differentiable at almost every $x\in\Omega$.

Then it turns out that $\Fom(u,\Omega)=+\infty$. 

\end{thm}
 
Let us discuss three consequences of this second result. The first one concerns the class of functions with bounded variation, without boundedness assumptions.

\begin{cor}[BV functions]\label{cor:BV}

Let $\omega:[0,+\infty)\to[0,+\infty)$ be a continuous function that satisfies (\ref{hp:omega>0}) and (\ref{hp:omega-int}). Let $\Omega\subseteq\re^{d}$ be a connected open set, and let $u:\Omega\to\re$ be a non-constant function with bounded variation.

Then it turns out that $\Fom(u,\Omega)=+\infty$. 

\end{cor}

The second consequence concerns the case where $\omega$ is bounded away from~0 at infinity, in which case we do not need to require the approximate differentiability of $u$. 

\begin{cor}[Case where $\omega$ is positive at infinity]\label{cor:limit}

Let $\omega:[0,+\infty)\to[0,+\infty)$ be a continuous function that satisfies (\ref{hp:omega>0}) and (\ref{hp:omega-int}), and in addition 
\begin{equation}
\liminf_{\mu\to +\infty}\omega(\mu)>0.
\label{hp:omega-liminf>0}
\end{equation}

Let $\Omega\subseteq\re^{d}$ be a connected open set, and let $u\in L^{\infty}(\Omega)$ be a non-constant function.

Then it turns out that $\Fom(u,\Omega)=+\infty$. 

\end{cor}

The last consequence of the second result is a positive answer to Question~\ref{open:ignat} with the additional assumption that $\omega$ is nondecreasing, but without any further requirement on $u$. The investigation of this case was suggested to the second author by L.~Ambrosio. 

\begin{cor}[Case where $\omega$ is nondecreasing]\label{cor:monotone}

Let $\omega:[0,+\infty)\to[0,+\infty)$ be a nondecreasing continuous function that satisfies (\ref{hp:omega>0}) and (\ref{hp:omega-int}). Let $\Omega\subseteq\re^{d}$ be a connected open set, and let $u:\Omega\to\re$ be a non-constant measurable function.

Then it turns out that $\Fom(u,\Omega)=+\infty$. 

\end{cor}

For the convenience of the reader, we summarize in Table~\ref{table:SoA} the positive and negative results that are known at the present. The positive answers of the last row follow from the $W^{1,1}$ case of~\cite[Theorem~1.5]{2005-ignat}, while the negative answers in the second column depend on the Heaviside counterexample of~~\cite[Theorem~1.2]{2005-ignat}. All remaining positive and negative answers follow from the results of the present paper. In addition, our counterexample and our positive results motivate some further open questions, which we state and discuss at the end of this paper. Some of these open problems appear as question marks in the table.

\begin{table}[ht]
\begin{center}

\renewcommand{\arraystretch}{1.6}
\begin{tabular}{|c||c|c|c|c|}
\hline
 & (P) & (P)+(I) & (P)+(I)+(L) & (P)+(I)+(M) 
\\
\hline\hline
measurable & No \cite{2005-ignat} & No (Thm.~\ref{thm:counterexample}) & No (Thm.~\ref{thm:counterexample}) & Yes (Cor.~\ref{cor:monotone}) 
\\
\hline
$L^{1}$ & No \cite{2005-ignat} & ??? & ??? & Yes (Cor.~\ref{cor:monotone}) 
\\
\hline
$L^{\infty}$ & No \cite{2005-ignat} & ??? & Yes (Cor.~\ref{cor:limit}) & Yes (Cor.~\ref{cor:monotone}) 
\\
\hline
$L^{\infty}+AD$ & No \cite{2005-ignat} & Yes (Thm.~\ref{thm:AD}) & Yes (Thm.~\ref{thm:AD}) & Yes (Thm.~\ref{thm:AD}) 
\\
\hline
$W^{1,1}$ & Yes \cite{2005-ignat} & Yes \cite{2005-ignat} & Yes \cite{2005-ignat} & Yes \cite{2005-ignat}
\\
\hline
\end{tabular} 
\caption{State of the art on Question~\ref{open:ignat}. Headers in the first column denote assumptions on $u$, with ``$AD$'' that stands for approximately differentiable at almost every point. Headers in the first row denote assumptions on $\omega$: (P) is the positivity requirement (\ref{hp:omega>0}), (I) is the integral condition (\ref{hp:omega-int}), (L) is the liminf condition (\ref{hp:omega-liminf>0}) at infinity, and (M) is the monotonicity of $\omega$. Questions marks refer to cases that are still open.}

\end{center}

\label{table:SoA}
\end{table}

\paragraph{\textmd{\emph{Overview of the technique}}}

In this paragraph we describe briefly and informally the ideas that are involved in our proofs.

The previous results by R.~Ignat show the direction that we have to follow in order to construct the counterexample of Theorem~\ref{thm:counterexample}. Indeed, we know that $\omega$ has to satisfy the necessary condition (\ref{hp:omega-int}), but not the sufficient condition (\ref{hp:omega-liminf}). Roughly speaking, this means that $\omega(\mu)$, as $\mu\to+\infty$, has to alternate regions where it is ``small'', so that (\ref{hp:omega-liminf}) fails, with regions where it is ``large'', so that the integral in (\ref{hp:omega-int}) diverges.

As for $u$, its difference quotients need to be ``large'' enough, because we know that $u\not\in W^{1,1}$, but in the same time these large difference quotients have to avoid the regions where $\omega$ is ``large'', because otherwise also $\Fom(u,\Omega)$ diverges. This anomalous concentration of difference quotients in alternating regions makes the construction of the counterexample challenging and somewhat counterintuitive.

As for the positive result of Theorem~\ref{thm:AD}, we begin by reducing ourselves to the one-dimensional case where $u$ is defined in some interval $(a,b)$. Then we distinguish two cases according to the values of the approximate derivative $\ap u'(x)$ in $(a,b)$.

\begin{itemize}

\item  If $\ap	u'(x)$ is different from~0 in a subset of $(a,b)$ with positive measure, then we can argue more or less as in the $W^{1,1}$ case (see~\cite[Theorem~1.5]{2005-ignat}). More precisely, it is enough to write $\Fom(u,(a,b))$ in the form
\begin{equation}
\int_{a}^{b}dx\int_{a}^{b}\omega\left(\frac{|u(y)-u(x)|}{|y-x|}\right)\frac{1}{|y-x|}\,dy,
\nonumber
\end{equation}
and then proving that the integral with respect to~$y$ diverges for every $x$ such that $\ap u'(x)\neq 0$. We refer to Theorem~\ref{thm:u'>0} for the details. 

\item  If $\ap u'(x)=0$ for almost every $x\in(a,b)$, then we need a completely different argument in order to exploit the integral condition (\ref{hp:omega-int}). The intuitive idea is to isolate, in the double integral (\ref{defn:F-omega}), the contribution of pairs $(x,y)$ such that the corresponding difference quotient is equal to a given $\mu$. This idea of expressing the integral of a function in terms of integrals over the level sets of another function (here the difference quotient) reminded us of the coarea formula and of the disintegration theorem in measure theory. Nevertheless, in this case there is not enough regularity in order to apply the former, and the latter is a nice general result that in its abstract form does not provide an efficient description of the measures on the fibers. In order to overcome this difficulty, we slice $(a,b)^{2}$ by considering sets of the form
\begin{equation}
Z(u,(a,b),[\mu-\delta,\mu]):=\left\{(x,y)\in(a,b)^{2}: y>x,\ 
\mu-\delta\leq\frac{u(y)-u(x)}{y-x}\leq\mu\right\},
\nonumber
\end{equation}
and we introduce the function
\begin{equation}
\gamma(\mu):=\liminf_{\delta\to 0^{+}}\frac{1}{\delta}\iint_{Z(u,(a,b),[\mu-\delta,\mu])}
\omega\left(\frac{|u(y)-u(x)|}{y-x}\right)\frac{1}{y-x}\,dx\,dy.
\label{defn:gamma(mu)}
\end{equation}

If $u$ admits two Lebesgue points $x_{1}<x_{2}$ with $u(x_{1})<u(x_{2})$ (the other case is symmetric), then we prove that
\begin{equation}
\gamma(\mu)\geq c_{0}\frac{\omega(\mu)}{\mu^{2}}
\qquad
\forall\mu>\mu_{0}
\label{th:gamma>omega}
\end{equation}
for suitable positive constants $c_{0}$ and $\mu_{0}$. When the integral condition (\ref{hp:omega-int}) is satisfied, this is enough to conclude that $\Fom(u,(a,b))=+\infty$. 

\end{itemize}

In each of the two cases discussed above only one assumption on $\omega$ is involved. More precisely, in the first case the divergence of $\Fom(u,(a,b))$ relies only on the positivity assumption (\ref{hp:omega>0}), while in the second case it relies only on the integral condition (\ref{hp:omega-int}).

The proof of (\ref{th:gamma>omega}) follows from Theorem~\ref{thm:loc-vs-glob}, which is the technical core of the paper, after introducing an auxiliary function. In the original variables, the disintegration with respect to $\mu$ corresponds to intersecting the graph of $u(x)$ with the family of parallel lines with equation $\mu x+a$. The geometric idea is that, when  the angular coefficient $\mu$ is large enough, then there are ``enough values'' of $a$ for which the line intersects the graph of $u(x)$ in at least two points. These multiple intersections suggest that the integration set in the right-hand side of (\ref{defn:gamma(mu)}) is big enough, even for small values of $\delta$. The quantification of this effect requires a further disintegration process, this time with respect to $a$, which leads to a representation of $\gamma(\mu)$ as an integral with respect to $a$. We refer to section~\ref{sec:loc-vs-glob} for the details.

We point out that our proof of Theorem~\ref{thm:loc-vs-glob} requires the approximate differentiability, but we do suspect that the conclusion could be true even without this assumption, which would lead to a proof of Theorem~\ref{thm:AD} without the approximate differentiability assumption (see Open problem~\ref{open:loc-vs-glob}).

\paragraph{\textmd{\emph{Structure of the paper}}}

This paper is organized as follows. In section~\ref{sec:counterexample} we present our counterexample that proves Theorem~\ref{thm:counterexample}. In section~\ref{sec:lemma} we collect some simple technical lemmata. In section~\ref{sec:loc-vs-glob} we develop the technical core of the paper, that culminates in the proof of Theorem~\ref{thm:loc-vs-glob}. In section~\ref{sec:proof-AD} we prove Theorem~\ref{thm:AD} and its corollaries. Finally, in section~\ref{sec:further} we present some questions that remain open.


\setcounter{equation}{0}
\section{Our counterexample (proof of Theorem~\ref{thm:counterexample})}\label{sec:counterexample}

In this section we prove Theorem~\ref{thm:counterexample} by exhibiting an explicit example of functions $\omega$ and $u$ with the required properties. More precisely, we actually construct a class of counterexamples, depending on two sequences of real numbers $\{k_{n}\}$ and $\{\mu_{n}\}$ that grow fast enough.

\paragraph{\textmd{\textit{Preliminaries -- The Cantor set}}}

Let $C\subseteq[0,1]$ denote the classical middle-third Cantor set. For every positive integer $n$, let $A_{n}\subseteq[0,1]$ denote the open set that is removed from $[0,1]$ in the $n$-th step of the construction of $C$. More formally, the sequence $\{A_{n}\}$ is defined by
\begin{equation}
A_{1}=\left(\frac{1}{3},\frac{2}{3}\right)
\qquad\mbox{and}\qquad
A_{n+1}=\frac{A_{n}}{3}\cup\left(\frac{2}{3}+\frac{A_{n}}{3}\right)
\quad
\forall n\geq 1,
\nonumber
\end{equation}
so that
\begin{equation}
[0,1]\setminus C=\bigcup_{n=1}^{\infty}A_{n}.
\nonumber
\end{equation}

In the sequel we need the following two properties of $\{A_{n}\}$.

\begin{itemize}

\item  (Lebesgue measure). If $\Leb (A_n)$ denotes the Lebesgue measure of $A_{n}$, then
\begin{equation}
\Leb (A_n)=\frac{1}{2}\left(\frac{2}{3}\right)^{n}
\qquad
\forall n\geq 1.
\label{th:Leb-aj}
\end{equation}

\item (Distance estimate). For every pair of positive integers $i<j$, the distance between $A_{i}$ and $A_{j}$ is $3^{-j}$, and in particular
\begin{equation}
|y-x|\geq\frac{1}{3^{j}}
\qquad
\forall x\in A_{i},\quad\forall y\in A_{j}.
\label{th:dist-ai-aj}
\end{equation}

\end{itemize}

\paragraph{\textmd{\textit{Definition of $\omega$}}}

Let us choose once for all two sequences of real numbers $\{k_n\}\subseteq [1,+\infty)$ and $\{\mu_n\}\subseteq [1,+\infty)$ such that
\begin{equation}
\mu_n\geq 3^{n} k_n
\qquad\mbox{and}\qquad 
k_{n+1}\geq k_n+\mu_n+3
\label{defn:kn-mun}
\end{equation}
for every $n\geq 1$, and
\begin{equation}
\sum_{n=1}^{\infty}n^{2}\left(\frac{2}{3}\right)^{n}\exp\left((\log\mu_{n})^{1/4}\right)<+\infty.
\label{defn:mun-series}
\end{equation}

For example, we can consider $k_n := 10^{n^2}$ and $\mu_n := 3^{n}\cdot 10^{n^2}$.

We observe that $\mu_{n+1}>\mu_{n}+3$ for every $n\geq 1$, and therefore we can consider the function $\omega:[0,+\infty)\to[0,+\infty)$ such that
\begin{itemize}

\item  $\omega(0)=0$,

\item  $\omega(\mu)=\mu^{2}$ if $\mu\in[\mu_{i}+1,\mu_{i}+2]$ for some integer $i\geq 1$,

\item  $\omega(\mu)=\exp((\log\mu)^{1/4})$ if $\mu\in[\mu_{i}+3,\mu_{i+1}]$ for some integer $i\geq 1$,

\item  $\omega$ is the affine interpolation between the values at the endpoints in $[0,\mu_{1}+1]$, and in all subsequent intervals of the form $[\mu_{i},\mu_{i}+1]$ or $[\mu_{i}+2,\mu_{i}+3]$.

\end{itemize}

From the definition it follows that $\omega(\mu)$ is continuous, null in the origin, and satisfies the positivity assumption (\ref{hp:omega>0}) and the growth condition (\ref{th:growth-omega}). Moreover it turns out that
\begin{equation}
\int_{1}^{+\infty}\frac{\omega(\mu)}{\mu^{2}}\,d\mu\geq
\sum_{i=1}^{\infty}\int_{\mu_{i}+1}^{\mu_{i}+2}\frac{\omega(\mu)}{\mu^{2}}\,d\mu=
\sum_{i=1}^{\infty}1,
\nonumber
\end{equation}
which proves that $\omega$ satisfies also the integral condition (\ref{hp:omega-int}).

\paragraph{\textmd{\textit{Definition of $u$}}}

Let $\{k_n\}$ and $\{\mu_n\}$ be the two sequences that we have chosen above. Let $u:[0,1]\to\re$ be defined by
\begin{equation}
u(x):=\left\{
\begin{array}{l@{\qquad}l}
0 & \mbox{if }x\in C, 
\\[0.5ex]
k_{n} & \mbox{if $x\in A_{n}$ for some $n\geq 1$.}
\end{array}\right.
\nonumber
\end{equation}

We observe that the definition of $u$ can be stated in an alternative way by relying on the representation of real numbers in base~3. Indeed, if we characterize the Cantor set as the set of real numbers in $[0,1]$ that do not require the digit~1 in order to be expressed as a ternary (base~3) fraction, then $u(x)=k_{n}$ if and only if $n$ is the position of the first digit~1 that is required in the ternary representation of $x$. The values of $u(x)$ for $x\in C$ are not relevant because $C$ has zero Lebesgue measure.

\paragraph{\textmd{\textit{Key property of difference quotients of $u$}}}

We claim that, for every pair of positive integers $i<j$, it turns out that
\begin{equation}
\frac{|u(y)-u(x)|}{|y-x|} \in [\mu_{j-1}+3 , \mu_j]
\qquad
\forall x\in A_{i},\quad\forall y\in A_{j},
\label{th:diff-q-u}
\end{equation}
which implies that the difference quotients of $u$ lie in the intervals where $\omega$ is ``small''.

To this end, we just observe that $u(y)-u(x)=k_{j}-k_{i}$ and
\begin{equation}
\frac{1}{3^{j}}\leq |y-x|\leq 1
\qquad
\forall x\in A_{i},\quad\forall y\in A_{j},
\nonumber
\end{equation}
where the estimate from below follows from (\ref{th:dist-ai-aj}). Thus from (\ref{defn:kn-mun}) we deduce that
\begin{equation}
\frac{|u(y)-u(x)|}{|y-x|}\geq 
k_{j}-k_{i}\geq
k_{j}-k_{j-1}\geq
\mu_{j-1}+3,
\nonumber
\end{equation}
and
\begin{equation}
\frac{|u(y)-u(x)|}{|y-x|}\leq 
3^{j}(k_{j}-k_{i})\leq
3^{j}k_{j}\leq
\mu_{j},
\nonumber
\end{equation}
which proves (\ref{th:diff-q-u}).

\paragraph{\textmd{\textit{Conclusion}}}

It remains to show that $\Fom(u,(0,1))<+\infty$.

To this end, we observe that in the double integral there is no contribution from pairs $(x,y)$ with either $x\in C$ or $y\in C$ (because the Lebesgue measure of the Cantor set is equal to~0), and there is no contribution from pairs $(x,y)$ with $x$ and $y$ in the same open set $A_{i}$ (because in this case the difference quotient is~0, and $\omega(0)=0$). Therefore, we can limit ourselves to pairs $(x,y)$ with $x\in A_{i}$ and $y\in A_{j}$ for some $i\neq j$. Due to the symmetry in $x$ and $y$ we obtain that
\begin{equation}
\Fom(u,(0,1))=2\sum_{j=2}^{\infty}\sum_{i=1}^{j-1}\iint_{A_{i}\times A_{j}}
\omega\left(\frac{|u(y)-u(x)|}{|y-x|}\right)\cdot\frac{1}{|y-x|}\,dx\,dy.
\label{eqn:double-sum}
\end{equation}

From (\ref{th:diff-q-u}), and the explicit formula for $\omega(\mu)$ in $[\mu_{j-1}+3,\mu_{j}]$, we deduce that
\begin{equation}
\omega\left(\frac{|u(y)-u(x)|}{|y-x|}\right)\leq
\exp\left((\log\mu_{j})^{1/4}\right)
\qquad
\forall (x,y)\in A_{i}\times A_{j}.
\label{est:omega-aj}
\end{equation}

On the other hand, from (\ref{th:dist-ai-aj}) we know that
\begin{equation}
A_{i}\subseteq\left[0,y-3^{-j}\right]\cup\left[y+3^{-j},1\right]
\qquad
\forall y\in A_{j},\quad\forall i<j,
\nonumber
\end{equation}
and therefore for every $y\in A_{j}$ it turns out that
\begin{equation}
\int_{A_{i}}\frac{1}{|y-x|}\,dx\leq
\int_{0}^{y-3^{-j}}\frac{1}{|y-x|}\,dx+\int_{y+3^{-j}}^{1}\frac{1}{|y-x|}\,dx\leq
2\log(3^{j})
\label{est:x-y-aj}
\end{equation}

From (\ref{est:omega-aj}), (\ref{est:x-y-aj}) and (\ref{th:Leb-aj}) we deduce that
\begin{eqnarray*}
\iint_{A_{i}\times A_{j}}
\omega\left(\frac{|u(y)-u(x)|}{|y-x|}\right)\cdot\frac{1}{|y-x|}\,dx\,dy
& \leq &
\exp\left((\log\mu_{j})^{1/4}\right)
\iint_{A_{i}\times A_{j}}\frac{1}{|y-x|}\,dx\,dy
\\[1ex]
& \leq & \exp\left((\log\mu_{j})^{1/4}\right)\cdot\Leb(A_{j})\cdot 2\log(3^{j})
\\[1ex]
& \leq & 2j\left(\frac{2}{3}\right)^{j}\exp\left((\log\mu_{j})^{1/4}\right)
\end{eqnarray*}
for every $i<j$. Plugging this estimate into (\ref{eqn:double-sum}) we conclude that
\begin{equation}
\Fom(u,(0,1))\leq
4\sum_{j=2}^{\infty}\sum_{i=1}^{j-1}
j\left(\frac{2}{3}\right)^{j}\exp\left((\log\mu_{j})^{1/4}\right)\leq
4\sum_{j=1}^{\infty}j^{2}\left(\frac{2}{3}\right)^{j}\exp\left((\log\mu_{j})^{1/4}\right).
\nonumber
\end{equation}

Due to (\ref{defn:mun-series}), this proves that $\Fom(u,(0,1))$ is finite.
\qed



\setcounter{equation}{0}
\section{Preliminaries}\label{sec:lemma}

In this section we state and prove some simple technical results that we need in the sequel. In the first one we show the divergence of the integral of $1/x$ over a measurable set $E\subseteq(0,+\infty)$ that has positive upper density in 0.

\begin{lemma}\label{lemma:pos-dens}

Let $E\subseteq(0,+\infty)$ be a measurable set with positive upper density in~0, namely such that
\begin{equation}
\limsup_{\delta\to 0^{+}}\frac{\Leb(E\cap(0,\delta))}{\delta}>0.
\label{hp:pos-dens}
\end{equation}

Then it turns out that
\begin{equation}
\int_{E}\frac{1}{x}\,dx=+\infty.
\nonumber
\end{equation}

\end{lemma}

\begin{proof}

From (\ref{hp:pos-dens}) we deduce that there exist a real number $\ep_{0}>0$ and a sequence $\delta_{n}\to 0^{+}$ such that
\begin{equation}
\Leb(E\cap(0,\delta_{n}))\geq\ep_{0}\delta_{n}
\qquad
\forall n\geq 1.
\label{hp:quantitative}
\end{equation}

Up to subsequences, we can also assume that
\begin{equation}
\delta_{n+1}\leq\frac{\ep_{0}}{2}\delta_{n}
\qquad
\forall n\geq 1.
\nonumber
\end{equation}

Let us set
\begin{equation}
E_{n}:=E\cap(\delta_{n+1},\delta_{n})
\qquad
\forall n\geq 1.
\nonumber
\end{equation}

Then $\{E_{n}\}$ is a sequence of \emph{disjoint} subsets of $E$ such that
\begin{equation}
\Leb(E_{n})\geq\frac{\ep_{0}}{2}\delta_{n}
\qquad
\forall n\geq 1,
\label{est:leb-An}
\end{equation}
because otherwise
\begin{equation}
\Leb(E\cap(0,\delta_{n}))=
\Leb(E\cap(0,\delta_{n+1}))+\Leb(E_{n})\leq
\delta_{n+1}+\Leb(E_{n})<
\ep_{0}\delta_{n},
\nonumber
\end{equation}
which contradicts (\ref{hp:quantitative}). From (\ref{est:leb-An}) we deduce that
\begin{equation}
\int_{E_{n}}\frac{1}{x}\,dx\geq
\Leb(E_{n})\cdot\frac{1}{\delta_{n}}=\frac{\ep_{0}}{2}
\qquad
\forall n\geq 1,
\nonumber
\end{equation}
and therefore
\begin{equation}
\int_{E}\frac{1}{x}\,dx\geq
\sum_{n=1}^{\infty}\int_{E_{n}}\frac{1}{x}\,dx\geq
\sum_{n=1}^{\infty}\frac{\ep_{0}}{2}=
+\infty,
\nonumber
\end{equation}
which completes the proof.
\end{proof}


In the second result we show that a function of class $C^{1}$ with nonzero derivative maps sets of positive measure into sets of positive measure.

\begin{lemma}\label{lemma:g(B)}

Let $g:\re\to\re$ be a function of class $C^{1}$, and let $B\subseteq\re$ be a measurable set such that $\Leb(B)>0$, and $g'(x)\neq 0$ for every $x\in B$.

Then it turns out that $\Leb(g(B))>0$.

\end{lemma}

\begin{proof}

To begin with, we observe that $g(B)$ is measurable, because it is the image of a measurable set through a locally Lipschitz function. Now let us choose a point $b_{0}\in B$ such that
\begin{equation}
\Leb(B\cap(b_{0}-r,b_{0}+r))>0
\qquad
\forall r>0
\nonumber
\end{equation}
(almost every point in $B$ has this property). Let us assume, without loss of generality, that $g'(b_{0})>0$ (the case $g'(b_{0})<0$ is symmetric). Then there exist $r_{0}>0$ and $\nu_{0}>0$ such that $g'(x)\geq\nu_{0}$ for every $x\in[b_{0}-r_{0},b_{0}+r_{0}]$, and a fortiori
\begin{equation}
g'(x)\geq\nu_{0}>0
\qquad
\forall x\in B_{0}:=B\cap(b_{0}-r_{0},b_{0}+r_{0}).
\nonumber
\end{equation}

The restriction of $g$ to the interval $[b_{0}-r_{0},b_{0}+r_{0}]$ is injective, and therefore from the area formula we conclude that
\begin{equation}
\Leb(g(B))\geq
\Leb(g(B_{0}))=
\int_{B_{0}}g'(x)\,dx\geq
\nu_{0}\cdot\Leb(B_{0})>
0,
\nonumber
\end{equation}
which completes the proof.
\end{proof}



The third lemma is a variant of the standard result concerning the derivation of measures (we omit the classical proof based on Vitali's covering theorem). This is the result that we need, but we recall that a stronger version with the upper density (namely with the limsup instead of the liminf) is also true.

\begin{lemma}\label{lemma:disintegration}

Let $\nu$ be a Borel measure in $\re$. Let us consider the lower density of $\nu$ with respect to the Lebesgue measure, defined by
\begin{equation}
\theta(\mu):=\liminf_{r\to 0^{+}}\frac{\nu([\mu-r,\mu])}{r}
\qquad
\forall\mu\in\re.
\nonumber
\end{equation}

Then the function $\theta(\mu)$ is measurable, and for every measurable set $A\subseteq\re$ it turns out that
\begin{equation}
\nu(A)\geq\int_{A}\theta(\mu)\,d\mu.
\nonumber
\end{equation}

\end{lemma}


The last result is a multi-variable inequality with some combinatorial flavor.

\begin{lemma}\label{lemma:olimpico}

Let $m\geq 2$ be an integer, and let $(r_{1},\ldots,r_{m})\in[0,1]^{m}$ be real numbers with sum $S\geq 1$.

Then it turns out that
\begin{equation}
\sum_{1\leq i<j\leq m}r_{i}r_{j}\geq S-1.
\label{th:olimpico}
\end{equation}

\end{lemma}

\begin{proof}

We argue by induction on $m$. Let us start with the case $m=2$ and let us assume, without loss of generality, that $r_{1}\geq r_{2}$. Since $S/2\leq r_{1}\leq 1$, and the function $\sigma\to\sigma(S-\sigma)$ is decreasing for $\sigma\geq S/2$, we obtain that
\begin{equation}
r_{1}r_{2}=r_{1}(S-r_{1})\geq 1(S-1),
\nonumber
\end{equation}
which proves (\ref{th:olimpico}) in this case.

Let us assume now that the conclusion holds true for some positive integer $m$, and let us consider $m+1$ real numbers $r_{1},\ldots,r_{m},r_{m+1}$ in $[0,1]$. Let us set $R:=r_{1}+\ldots+r_{m}$, and let us distinguish two cases.

\begin{itemize}

\item  If $R\leq 1$, then considering only the products with $j=m+1$ we find that
\begin{equation}
\sum_{1\leq i<j\leq m+1}r_{i}r_{j}\geq
R\cdot r_{m+1}=
R(S-R)\geq
S-1,
\nonumber
\end{equation}
where the last inequality follows from the case $m=2$.

\item  If $R\geq 1$, then we apply the inductive assumption to the sum of products that do not involve $r_{m+1}$, and we obtain that
\begin{equation}
\sum_{1\leq i<j\leq m+1}r_{i}r_{j}=
R\cdot r_{m+1}+\sum_{1\leq i<j\leq m}r_{i}r_{j}\geq
r_{m+1}+(R-1)=
S-1,
\nonumber
\end{equation}
which proves (\ref{th:olimpico}) also in the second case.
\qedhere

\end{itemize}

\end{proof}


\setcounter{equation}{0}
\section{Local expansion vs global contraction}\label{sec:loc-vs-glob}

Let $A\subseteq\re$ be a subset, and let $\varphi:A\to\re$ be a function. For every pair $(x,y)\in A^{2}$ with $x<y$, let
\begin{equation}
R\varphi(x,y):=\frac{\varphi(y)-\varphi(x)}{y-x}
\label{defn:dq}
\end{equation}
denote the corresponding difference quotient. For every subset $E\subseteq\re$ we consider the set
\begin{equation}
Z(\varphi,A,E):=\left\{(x,y)\in A^{2}:x<y,\ R\varphi(x,y)\in E\right\}
\label{defn:Z}
\end{equation}
of all pairs of points in $A$ that give rise to a difference quotient in $E$. The following result is the technical core of this paper.

\begin{thm}[Local expansion vs global contraction]\label{thm:loc-vs-glob}

Let $A\subseteq\re$ be a bounded measurable set, and let $\varphi:A\to\re$ be a measurable function.

Let us assume that the approximate derivative of $\varphi$ exists for almost every $x\in A$ and
\begin{equation}
|\ap\varphi'(x)|\geq 1
\qquad
\text{for almost every }x\in A.
\label{hp:loc-exp}
\end{equation}

Then it turns out that
\begin{equation}
\liminf_{\delta\to 0^{+}}\frac{1}{\delta}\iint_{Z(\varphi,A,[0,\delta])}\frac{1}{y-x}\,dx\,dy\geq
\Leb(A)-\Leb_{*}(\varphi(A)),
\label{th:loc-vs-glob}
\end{equation}
where the integration set $Z(\varphi,A,[0,\delta])$ is defined by (\ref{defn:Z}), and
\begin{equation}
\Leb_{*}(\varphi(A)):=\sup\{\Leb(K): K\text{ compact, }K\subseteq\varphi(A)\}
\nonumber
\end{equation}
denotes the inner Lebesgue measure of the image $\varphi(A)$.

\end{thm}

\begin{rmk}
\begin{em}

Theorem~\ref{thm:loc-vs-glob} above is significant only when the right-hand side of (\ref{th:loc-vs-glob}) is positive, namely when $\Leb_{*}(\varphi(A))<\Leb(A)$. In this case we say that $\varphi$ is a \emph{global contraction}, in the sense that the measure of the image is less than the measure of the domain, and a \emph{local expansion} because of assumption~(\ref{hp:loc-exp}).

The intuitive idea, supported by the area formula, is that the combination of the global contraction and of the local expansion assumptions should induce some sort of non-injectivity of $\varphi$, which in turn should force the set $Z(\varphi,A,[0,\delta])$ to be large enough even for small values of $\delta$. The toy example is the case where $A:=(0,1+\ell)$ for some $\ell\in(0,1)$, and
\begin{equation}
\varphi(x):=
\begin{cases}
x      & \text{if }x\in(0,1), \\
x-1    & \text{if }x\in(1,1+\ell).
\end{cases}
\nonumber
\end{equation}

In this case it turns out that $\varphi(A)=(0,1)$, and a direct computation shows that we have equality in (\ref{th:loc-vs-glob}).

\end{em}
\end{rmk}

In the sequel of this section we develop the theory that leads to a proof of Theorem~\ref{thm:loc-vs-glob} above.


\subsection{The cumulative distribution}

Let us introduce some notation. Let $g:\re\to\re$ be a function of class $C^{1}$, and let $K\subseteq\re$ be a compact set. For every $z\in\re$ we consider the \emph{level set}
\begin{equation}
L_{g,K}(z):=\{x\in K:g(x)=z\}.
\label{defn:level-set}
\end{equation}

For every interval $I\subseteq\re$ we call \emph{cumulative distribution} the function $M_{g,K}(I,z)$ defined by
\begin{equation}
M_{g,K}(I,z):=\Leb(\{x\in I\cap K:g(x)\leq z\})
\qquad \forall z\in\re.
\label{defn:M(I,z)}
\end{equation}

Finally we consider the set $K_{0}$ of all points in $K$ whose density with respect to $K$ is not equal to~1, and we define the set of \emph{regular values} as
\begin{equation}
\reg(g,K):=\re\setminus g(K_{0}).
\label{defn:reg}
\end{equation}

We observe that
\begin{equation}
\Leb(\re\setminus\reg(g,K))=\Leb(g(K_{0}))=0
\nonumber
\end{equation}
because it is well-known that $\Leb(K_{0})=0$, and Lipschitz functions (in this case the restriction of $g$ to any bounded set) map sets with zero Lebesgue measure into sets with zero Lebesgue measure.

In the following result we collect the main properties of the function $M_{g,K}(I,z)$.

\begin{prop}[Properties of the cumulative distribution]\label{prop:M(I,z)}

Let $g:\re\to\re$ be a function of class $C^{1}$, and let $K\subseteq\re$ be a compact set such that
\begin{equation}
g'(x)\neq 0
\qquad
\forall x\in K.
\label{hp:g'-neq-0}
\end{equation}

Then the set $\reg(g,K)$ defined by (\ref{defn:reg}), the level set $L_{g,K}(z)$ defined by (\ref{defn:level-set}), the cumulative distribution $M_{g,K}(I,z)$ defined by (\ref{defn:M(I,z)}), and its derivative $M_{g,K}'(I,z)$ with respect to $z$ have the following properties. 
\begin{enumerate}
\renewcommand{\labelenumi}{(\arabic{enumi})}

\item  For every $z\in\re$ the level set $L_{g,K}(z)$ is finite (possibly empty), and in particular its Lebesgue measure is equal to zero.

\item  For every interval $I\subseteq\re$ the function $z\mapsto M_{g,K}(I,z)$ is absolutely continuous.

\item  If for some $z\in\re$ the closure of the interval $I$ does not intersect the level set $L_{g,K}(z)$, then $M_{g,K}'(I,z)=0$.

\item  If $x_{0}\in L_{g,K}(z)$ for some $z\in\reg(g,K)$, and $I$ is an open interval that contains $x_{0}$, but whose closure does not contain any other element of $L_{g,K}(z)$, then
\begin{equation}
M_{g,K}'(I,z)=\frac{1}{|g'(x_{0})|}.
\label{th:M-single}
\end{equation}

\item  If the open interval $I$ contains $K$, and for some $z\in\reg(g,K)$ the level set $L_{g,K}(z)$ consists of $m$ distinct points $x_{1}$, \ldots, $x_{m}$ , then
\begin{equation}
M_{g,K}'(I,z)=\sum_{i=1}^{m}\frac{1}{|g'(x_{i})|}.
\label{th:M'-glob}
\end{equation}

\end{enumerate}

\end{prop}

\begin{proof}

For every interval $I\subseteq\re$ and every subset $E\subseteq\re$ we introduce the set
\begin{equation}
S_{g,K}(I,E):=\{x\in K\cap I:g(x)\in E\},
\nonumber
\end{equation}
and we observe that
\begin{equation}
M_{g,K}'(I,z)=
\lim_{h\to 0^{+}}\frac{\Leb(S_{g,K}(I,[z,z+h]))}{h}=
\lim_{h\to 0^{+}}\frac{\Leb(S_{g,K}(I,[z-h,h]))}{h},
\label{eqn:M'-S}
\end{equation}
provided that the two limits exist. Now we prove the five statements separately.

\paragraph{\textmd{\textit{Statement~(1)}}}

Let us assume by contradiction that the level set $L_{g,K}(z)$ is infinite for some $z\in\re$. Then it admits an accumulation point, which belongs to the same level set. In this point the derivative of $g$ has to vanish, which contradicts assumption (\ref{hp:g'-neq-0}).

\paragraph{\textmd{\textit{Statement~(2)}}}

We observe that the function $z\mapsto M_{g,K}(I,z)$ is non-decreasing, and its derivative is the measure $\nu$ such that $\nu(A)=\Leb(g^{-1}(A)\cap K\cap I)$ for every measurable set $A\subseteq\re$. 

If we assume by contradiction that $M_{g,K}(I,z)$ is not absolutely continuous, then there exists a measurable set $A\subseteq\re$ with $\Leb(A)=0$ such that the set $B:=g^{-1}(A)\cap K\cap I$ has positive Lebesgue measure. This set $B$ contradicts Lemma~\ref{lemma:g(B)}. 

\paragraph{\textmd{\textit{Statement~(3)}}}

We prove a stronger statement, namely that the set 
\begin{equation}
S_{g,K}(I,[z-h,z+h])=\{x\in I\cap K:z-h\leq g(x)\leq z+h\}
\nonumber
\end{equation}
is empty when $h>0$ is small enough. 

Indeed, if we assume by contradiction that this is not the case, then we find a sequence $h_{n}\to 0^{+}$ and a sequence $\{x_{n}\}\subseteq K\cap I$ such that $z-h_{n}\leq g(x_{n})\leq z+h_{n}$ for every $n\geq 1$. Up to subsequences, we can assume that $x_{n}$ converges to some point $x_{\infty}$. This point lies both in $K$ and in the closure of $I$, and satisfies $g(x_{\infty})=z$. But this contradicts the assumption that the closure of $I$ does not intersect the level set $L_{g,K}(z)$.

\paragraph{\textmd{\textit{Statement~(4)}}}

We assume without loss of generality that $g'(x_{0})>0$ (the other case is symmetric), and we show that for every $\ep\in(0,g'(x_{0}))$ it turns out that
\begin{equation}
\liminf_{h\to 0^{+}}\frac{S_{g,K}(I,[z,z+h])}{h}\geq
\frac{1}{g'(x_{0})+\ep},
\label{th:M'>=}
\end{equation}
and
\begin{equation}
\limsup_{h\to 0^{+}}\frac{S_{g,K}(I,[z,z+h])}{h}\leq
\frac{1}{g'(x_{0})-\ep}.
\label{th:M'<=}
\end{equation}

Since analogous inequalities can be proved also for $S_{g,K}(I,[z-h,z])$, letting $\ep\to 0^{+}$ this implies (\ref{th:M-single}) because of (\ref{eqn:M'-S}). Now we prove the two estimates separately.

\begin{itemize}

\item  The key point in the proof of (\ref{th:M'>=}) is showing that
\begin{equation}
S_{g,K}(I,[z,z+h])\supseteq 
K\cap\left(x_{0},x_{0}+\frac{h}{g'(x_{0})+\ep}\right)
\label{th:inclusion>=}
\end{equation}
when $h$ is small enough. Indeed, from this inclusion it follows that
\begin{eqnarray*}
\liminf_{h\to 0^{+}}\frac{\Leb(S_{g,K}(I,[z,z+h]))}{h} & \geq &
\lim_{h\to 0^{+}}
\frac{\Leb\left(K\cap\left(x_{0},x_{0}+\dfrac{h}{g'(x_{0})+\ep}\right)\right)}{h}
\\[1ex]
& = &
\frac{1}{g'(x_{0})+\ep},
\end{eqnarray*}
where the last equality follows from the fact that $x_{0}$ has density~1 with respect to $K$, which in turn follows from the assumption that $z\in\reg(g,K)$.

In order to prove (\ref{th:inclusion>=}) we exploit that $g$ is of class $C^{1}$, and in particular there exists $r_{0}>0$ such that $(x_{0},x_{0}+r_{0})\subseteq I$ and
\begin{equation}
g(x_{0})\leq g(x)\leq g(x_{0})+(g'(x_{0})+\ep)(x-x_{0})
\qquad
\forall x\in(x_{0},x_{0}+r_{0}).
\nonumber
\end{equation}

It follows that, when $h<(g'(x_{0})+\ep)r_{0}$, it turns out that
\begin{equation}
z\leq g(x)\leq z+(g'(x_{0})+\ep)\frac{h}{g'(x_{0})+\ep}=z+h
\nonumber
\end{equation}
for every
\begin{equation}
x\in\left(x_{0},x_{0}+\frac{h}{g'(x_{0})+\ep}\right),
\nonumber
\end{equation}
which implies (\ref{th:inclusion>=}).

\item  In order to prove (\ref{th:M'<=}) it is enough to show that
\begin{equation}
S_{g,K}(I,[z,z+h])\subseteq 
\left[x_{0},x_{0}+\frac{h}{g'(x_{0})-\ep}\right]
\nonumber
\end{equation}
when $h>0$ is small enough. 

To this end, let us assume by contradiction that this is not the case. Then there exist a sequence $h_{n}\to 0^{+}$, and a sequence $\{x_{n}\}\subseteq K\cap I$ that for every $n\geq 1$ satisfies $z\leq g(x_{n})\leq z+h_{n}$ and exactly one of the following two inequalities
\begin{equation}
x_{n}>x_{0}+\frac{h_{n}}{g'(x_{0})-\ep},
\qquad\qquad
x_{n}<x_{0}.
\label{xn-x0}
\end{equation}

Up to subsequences, we can assume that $x_{n}$ satisfies always the same inequality in (\ref{xn-x0}), and that $x_{n}$ has a limit $x_{\infty}$. This limit lies in $K$ and in the closure of $I$, and in addition $g(x_{\infty})=z$. On the other hand, we know that the unique element of the level set $L_{g,K}(z)$ that belongs to the closure of $I$ is $x_{0}$, and therefore necessarily $x_{n}\to x_{0}$. At this point (\ref{xn-x0}) implies that one of the following two relations
\begin{equation}
\frac{g(x_{n})-g(x_{0})}{x_{n}-x_{0}}\leq
\frac{h_{n}}{x_{n}-x_{0}}<
g'(x_{0})-\ep,
\qquad\quad
\frac{g(x_{n})-g(x_{0})}{x_{n}-x_{0}}<0
\nonumber
\end{equation}
is true for every $n\geq 1$. Since the left-hand sides tend to $g'(x_{0})$ as $n\to +\infty$, both of them lead to an absurd (we assumed at the beginning that $g'(x_{0})>0$).

\end{itemize}

\paragraph{\textmd{\textit{Statement~(5)}}}

It is enough to write $I$ as the union of a finite number of subintervals, each of which fits into the framework of either statement~(3) or statement~(4).
\end{proof}


\subsection{A representation of the liminf in terms of the cumulative distribution}

In this subsection we estimate from below the
\begin{equation}
\liminf_{\delta\to 0^{+}}\frac{1}{\delta}
\iint_{Z(g,K,[0,\delta])}\frac{1}{y-x}\,dx\,dy
\label{defn:int-g-K}
\end{equation}
with an integral in one variable that involves the derivative of the cumulative density. The result follows from a sort of disintegration of the double integral in (\ref{defn:int-g-K}) with respect to the function $g(x)$. To this end, for every subset $E\subseteq\re$ we consider the set 
\begin{equation}
Z(g,K,[0,\delta],E):=\{(x,y)\in Z(g,K,[0,\delta]):g(x)\in E\},
\nonumber
\end{equation}
and we estimate from below the asymptotic behavior of the integral of $1/(y-x)$ over this set when $E$ is a small interval.

\begin{prop}[Density estimate for the disintegration]\label{prop:disintegration-1}

Let $g:\re\to\re$ be a function of class $C^{1}$, and let $K\subseteq\re$ be a compact set satisfying (\ref{hp:g'-neq-0}).

Let $c\in(0,1)$, and let $z\in\reg(g,K)$. Let us assume that the level set $L_{g,K}(z)$ consists of $m$ elements $x_{1}<\ldots<x_{m}$, with $m\geq 2$. 

Then there exists a real number $\delta_{0}>0$ such that
\begin{equation}
\liminf_{\ep\to 0^{+}}\frac{1}{\ep}
\iint_{Z(g,K,[0,\delta],[z-\ep,z])}\frac{1}{y-x}\,dx\,dy\geq
(1-c)^{3}P(z)\delta
\qquad
\forall\delta\in(0,\delta_{0}),
\label{est:main-delta-0}
\end{equation}
where
\begin{equation}
P(z):=\sum_{1\leq i<j\leq m}\frac{1}{|g'(x_{i})|}\cdot\frac{1}{|g'(x_{j})|}.
\label{defn:P(z)}
\end{equation}

\end{prop}

\begin{proof}

Let us set
\begin{equation}
\tau:=\min\{x_{j}-x_{i}:1\leq i<j\leq m\},
\nonumber
\end{equation}
and let us choose a positive real number $r$ such that
\begin{equation}
\frac{4r}{\tau}\leq c.
\label{defn:r}
\end{equation}

This condition implies in particular that $r<\tau/4$, and hence the intervals $[x_{i}-r,x_{i}+r]$ are pairwise disjoint for $i=1,\ldots,m$. 

For every pair of indices $1\leq i<j\leq m$ we consider the product set 
\begin{eqnarray*}
R_{i,j}(\ep,\delta,r) & := &
S_{g,K}((x_{i}-r,x_{i}+r),[z-\ep,z])\times
\\[0.5ex]
& & 
\mbox{}\times S_{g,K}((x_{j}-r,x_{j}+r),[z,z+(1-c)\delta(x_{j}-x_{i}-2r)]).
\end{eqnarray*}

We observe that these sets are pairwise disjoint, and we claim that they satisfy
\begin{equation}
R_{i,j}(\ep,\delta,r)\subseteq
Z(g,K,[0,\delta],[z-\ep,z])
\qquad
\forall\ep\in(0,c\delta(\tau-2r)).
\nonumber
\end{equation}

To this end, the only thing to check is that the difference quotients $Rg(x,y)$ of the function $g$ satisfy $0\leq Rg(x,y)\leq\delta$ for every $(x,y)\in R_{i,j}(\ep,\delta,r)$. The estimate from below is true because $g(y)\geq z\geq g(x)$ for every admissible value of $x$ and $y$. The estimate from above follows from the chain of inequalities
\begin{eqnarray*}
g(y)-g(x) & = &
(g(y)-z)-(g(x)-z)
\\
& \leq &
(1-c)\delta(x_{j}-x_{i}-2r)+\ep
\\
& \leq &
(1-c)\delta(x_{j}-x_{i}-2r)+c\delta(\tau-2r)
\\
& \leq &
(1-c)\delta(y-x)+c\delta(y-x)
\\
& = &
\delta(y-x).
\end{eqnarray*}

At this point we know that
\begin{equation}
\iint_{Z(g,K,[0,\delta],[z-\ep,z])}\frac{1}{y-x}\,dx\,dy\geq
\sum_{1\leq i<j\leq m}\iint_{R_{i,j}(\ep,\delta,r)}\frac{1}{y-x}\,dx\,dy,
\nonumber
\end{equation}
and therefore it is enough to show that, when $\delta>0$ is small enough, the estimate
\begin{equation}
\liminf_{\ep\to 0^{+}}\frac{1}{\ep}\iint_{R_{i,j}(\ep,\delta,r)}\frac{1}{y-x}\,dx\,dy\geq
(1-c)^{3}\delta\frac{1}{|g'(x_{i})|}\cdot\frac{1}{|g'(x_{j})|}
\label{th:int-Rij}
\end{equation}
holds true for every $1\leq i<j\leq m$. In order to estimate the integral, we observe that (\ref{defn:r}) implies that
\begin{equation}
\frac{x_{j}-x_{i}-2r}{x_{j}-x_{i}+2r}=
1-\frac{4r}{x_{j}-x_{i}+2r}\geq
1-\frac{4r}{\tau}\geq
1-c,
\nonumber
\end{equation}
and therefore
\begin{equation}
\frac{1}{y-x}\geq
\frac{1}{x_{j}-x_{i}+2r}\geq
\frac{1-c}{x_{j}-x_{i}-2r}
\qquad
\forall (x,y)\in R_{i,j}(\ep,\delta,r).
\nonumber
\end{equation}

Since the two dimensional Lebesgue measure of $R_{i,j}(\ep,\delta,r)$ is 
\begin{eqnarray*}
\Leb^{2}(R_{i,j}(\ep,\delta,r)) & = &
\Leb\left(S_{g,K}((x_{i}-r,x_{i}+r),[z-\ep,z])\strut\right)\cdot
\\[0.5ex]
& & 
\mbox{}\cdot\Leb\left(S_{g,K}((x_{j}-r,x_{j}+r),[z,z+(1-c)\delta(x_{j}-x_{i}-2r)])\strut\right),
\end{eqnarray*}
we obtain that
\begin{equation}
\frac{1}{\ep}\iint_{R_{i,j}(\ep,\delta,r)}\frac{1}{y-x}\,dx\,dy\geq
(1-c)\cdot G(\ep)\cdot H(\delta)\cdot(1-c)\delta,
\nonumber
\end{equation}
where
\begin{equation}
G(\ep):=\frac{\Leb\left(S_{g,K}((x_{i}-r,x_{i}+r),[z-\ep,z])\strut\right)}{\ep},
\nonumber
\end{equation}
and
\begin{equation}
H(\delta):=\frac{\Leb\left(S_{g,K}((x_{j}-r,x_{j}+r),[z,z+(1-c)\delta(x_{j}-x_{i}-2r)])\strut\right)}{(1-c)\delta(x_{j}-x_{i}-2r)}.
\nonumber
\end{equation}

Now from statement~(4) of Proposition~\ref{prop:M(I,z)} we deduce that
\begin{equation}
\lim_{\delta\to 0}H(\delta)=
M_{g,K}'((x_{j}-r,x_{j}+r),z)=
\frac{1}{|g'(x_{j})|},
\nonumber
\end{equation}
and therefore when $\delta$ is small enough it turns out that
\begin{equation}
H(\delta)\geq\frac{1-c}{|g'(x_{j})|}.
\nonumber
\end{equation}

More precisely, there exists $\delta_{0}>0$ such that the estimate
\begin{equation}
\frac{1}{\ep}\iint_{R_{i,j}(\ep,\delta,r)}\frac{1}{y-x}\,dx\,dy\geq
(1-c)\cdot G(\ep)\cdot\frac{1-c}{|g'(x_{j})|}\cdot(1-c)\delta
\nonumber
\end{equation}
holds true for every $\delta\in(0,\delta_{0})$, every $1\leq i<j\leq m$, and every $\ep\in(0,c\delta(\tau-2r))$. At this point we can let $\ep\to 0^{+}$, and exploiting again statement~(4) of Proposition~\ref{prop:M(I,z)} we conclude that
\begin{equation}
\lim_{\ep\to 0}G(\ep)=
M_{g,K}'((x_{i}-r,x_{i}+r),z)=
\frac{1}{|g'(x_{i})|},
\nonumber
\end{equation}
which completes the proof of (\ref{th:int-Rij}).
\end{proof}


When we let $\delta\to 0$ in Proposition~\ref{prop:disintegration-1} we obtain the following estimate from below for (\ref{defn:int-g-K}).

\begin{prop}[Liminf vs cumulative distribution]\label{prop:disintegration-2}

Let $g:\re\to\re$ be a function of class $C^{1}$, and let $K\subseteq\re$ be a compact set such that
\begin{equation}
|g'(x)|\geq 1
\qquad
\forall x\in K.
\label{hp:g'>=1}
\end{equation}

Let $I\subseteq\re$ be an open interval that contains $K$, and let us consider the set $\reg(g,K)$ defined by (\ref{defn:reg}) and the function $M_{g,K}(I,z)$ defined by (\ref{defn:M(I,z)}).

Then it turns out that
\begin{equation}
\liminf_{\delta\to 0^{+}}
\frac{1}{\delta}\iint_{Z(g,K,[0,\delta])}\frac{1}{y-x}\,dx\,dy\geq
\int_{\reg(g,K)}[M_{g,K}'(I,z)-1]_{+}\,dz,
\nonumber
\end{equation}
where $[\alpha]_{+}:=\max\{\alpha,0\}$ denotes the positive part of $\alpha$.

\end{prop}

\begin{proof}

The proof follows from the combination of two estimates. In Step~1 we show that
\begin{equation}
\liminf_{\delta\to 0^{+}}
\frac{1}{\delta}\iint_{Z(g,K,[0,\delta])}\frac{1}{y-x}\,dx\,dy\geq
\int_{\reg(g,K)}P(z)\,dz,
\label{th:liminf-P(z)}
\end{equation}
where $P(z)$ is defined by (\ref{defn:P(z)}) if the level set $L_{g,K}(z)$ has at least two elements, and $P(z):=0$ otherwise. In Step~2 we show that, under assumption (\ref{hp:g'>=1}), it turns out that
\begin{equation}
P(z)\geq[M_{g,K}'(I,z)-1]_{+}
\qquad
\forall z\in\reg(g,K)
\label{th:P>=M}
\end{equation}
for every open interval $I\subseteq\re$ that contains $K$.

\paragraph{\textmd{\textit{Step~1}}}

Let us prove (\ref{th:liminf-P(z)}). To this end, for every $c\in(0,1)$ and every $\delta_{0}>0$, we introduce the set $\reg(g,K,c,\delta_{0})$ of all $z\in\reg(g,K)$ such that (\ref{est:main-delta-0}) holds true. We spare the reader the long but rather standard verification that this set is actually measurable. Then we introduce the measure $\nu_{\delta}$ in $\re$ defined by
\begin{equation}
\nu_{\delta}(M):=\frac{1}{\delta}\iint_{Z(g,K,[0,\delta],M)}\frac{1}{y-x}\,dx\,dy
\nonumber
\end{equation}
for every Borel subset $M\subseteq\re$, and we interpret (\ref{est:main-delta-0}) as an estimate from below for the density of $\nu_{\delta}$ with respect to the Lebesgue measure. Thus from  Lemma~\ref{lemma:disintegration} we deduce that
\begin{equation}
\frac{1}{\delta}\iint_{Z(g,K,[0,\delta])}\frac{1}{y-x}\,dx\,dy=
\nu_{\delta}(\re)\geq
(1-c)^{3}\int_{\reg(g,K,c,\delta_{0})}P(z)\,dz
\nonumber
\end{equation}
for every $\delta\in(0,\delta_{0})$. Since the right-hand side does not depend on $\delta$, letting $\delta\to 0^{+}$ in the left-hand side we deduce that
\begin{equation}
\liminf_{\delta\to 0^{+}}
\frac{1}{\delta}\iint_{Z(g,K,[0,\delta])}\frac{1}{y-x}\,dx\,dy\geq
(1-c)^{3}\int_{\reg(g,K,c,\delta_{0})}P(z)\,dz
\nonumber
\end{equation}
for every admissible value of $c$ and $\delta_{0}$.
On the other hand, from Proposition~\ref{prop:disintegration-1} we know that, for every fixed  $c\in(0,1)$, every value $z\in\reg(g,K)$ belongs to $\reg(g,K,c,\delta_{0})$ when $\delta_{0}$ is small enough. Therefore, letting $\delta_{0}\to 0^{+}$ we obtain that
\begin{equation}
\liminf_{\delta\to 0^{+}}
\frac{1}{\delta}\iint_{Z(g,K,[0,\delta])}\frac{1}{y-x}\,dx\,dy\geq
(1-c)^{3}\int_{\reg(g,K)}P(z)\,dz.
\nonumber
\end{equation}

Since $c\in(0,1)$ is arbitrary, this proves (\ref{th:liminf-P(z)}).

\paragraph{\textmd{\textit{Step~2}}}

Let us prove (\ref{th:P>=M}). To begin with, we observe that the inequality is trivial if $M_{g,K}'(I,z)\leq 1$. Otherwise, from (\ref{th:M'-glob}) and assumption (\ref{hp:g'>=1}), we deduce that the level set $L_{g,K}(z)$ has at least two elements, in which case $P(z)$ is given by (\ref{defn:P(z)}). 

At this point we can apply Lemma~\ref{lemma:olimpico} with $r_{i}:=|g'(x_{i})|^{-1}$ and exploiting again (\ref{th:M'-glob}) we obtain (\ref{th:P>=M}).
\end{proof}


\subsection{Proof of Theorem~\ref{thm:loc-vs-glob}}

\paragraph{\textmd{\textit{Choice of a smooth function and a compact set}}}

Let $\eta$ be a positive real number. We claim that there exist a compact set $K\subseteq A$ and a function $g\in C^{1}(\re)$ such that
\begin{gather}
\Leb(K)\geq\Leb(A)-\eta,
\label{est:Leb-K}
\\
g(x)=\varphi(x)
\qquad
\forall x\in K,
\label{eqn:g=phi}
\\
|g'(x)|\geq 1
\qquad
\forall x\in K.
\label{eqn:g'-K}
\end{gather}

The construction of $K$ requires two steps. To begin with, we apply the result of~\cite[Theorem~3.1.16]{federer:GMT} and we find a compact set $\hatK\subseteq A$ and a function $g\in C^{1}(\re)$ such that $\Leb(A\setminus\hatK)\leq\eta/2$ and $\varphi(x)=g(x)$ for every $x\in\hatK$. Then we consider the set $\hatK_{1}$ of all points in $\hatK$ that have density~1 with respect to $\hatK$ (but a positive upper density would be enough), and the set $A_{1}$ of points in $A$ where $|\ap\varphi'(x)|\geq 1$.

We observe that $\Leb(\hatK)=\Leb(\hatK_{1}\cap A_{1})$, and that
\begin{equation}
g'(x)=\ap\varphi'(x)
\qquad
\forall x\in\hatK_{1}\cap A_{1},
\nonumber
\end{equation}
and in particular $|g'(x)|\geq 1$ in this set. At this point we conclude by choosing a compact set $K\subseteq\hatK_{1}\cap A_{1}$ such that $\Leb((\hatK_{1}\cap A_{1})\setminus K)\leq\eta/2$.

\paragraph{\textmd{\textit{Conclusion}}}

From (\ref{eqn:g=phi}) we deduce that $Z(\varphi,A,[0,\delta])\supseteq Z(g,K,[0,\delta])$, and hence
\begin{equation}
\iint_{Z(\varphi,A,[0,\delta])}\frac{1}{y-x}\,dx\,dy\geq
\iint_{Z(g,K,[0,\delta])}\frac{1}{y-x}\,dx\,dy
\qquad
\forall\delta>0.
\nonumber
\end{equation}

Thanks to (\ref{eqn:g'-K}) we can apply Proposition~\ref{prop:disintegration-2} and deduce that
\begin{equation}
\liminf_{\delta\to 0^{+}}\frac{1}{\delta}
\iint_{Z(g,K,[0,\delta])}\frac{1}{y-x}\,dx\,dy\geq
\int_{\reg(g,K)}[M_{g,K}'(I,z)-1]_{+}\,dz,
\nonumber
\end{equation}
where $I$ is any open interval that contains $K$. In order to compute the last integral we observe that
\begin{equation}
\int_{\reg(g,K)}[M_{g,K}'(I,z)-1]_{+}\,dz=
\int_{\re}[M_{g,K}'(I,z)-1]_{+}\,dz=
\int_{g(K)}[M_{g,K}'(I,z)-1]_{+}\,dz,
\nonumber
\end{equation}
where the first equality follows from $\Leb(\re\setminus\reg(g,K))=0$, and the second equality is true because $M_{g,K}'(I,z)=0$ when $z\not\in g(K)$. It remains to estimate the last integral, for which we obtain that
\begin{equation}
\int_{g(K)}[M_{g,K}'(I,z)-1]_{+}\,dz\geq
\int_{g(K)}(M_{g,K}'(I,z)-1)\,dz=
\int_{g(K)}M_{g,K}'(I,z)\,dz-\Leb(g(K)).
\nonumber
\end{equation}

Finally, we choose any interval $(-R,R)$ that is big enough to contain $g(K)$ and, exploiting again that $M_{g,K}'(I,z)=0$ when $z\not\in g(K)$, we deduce that
\begin{equation}
\int_{g(K)}M_{g,K}'(I,z)\,dz=
\int_{-R}^{R}M_{g,K}'(I,z)\,dz=
M_{g,K}(I,R)-M_{g,K}(I,-R)=
\Leb(K).
\nonumber
\end{equation} 

Putting things together we have proved that
\begin{equation}
\liminf_{\delta\to 0^{+}}\frac{1}{\delta}
\iint_{Z(\varphi,A,[0,\delta])}\frac{1}{y-x}\,dx\,dy\geq
\Leb(K)-\Leb(g(K))\geq
\Leb(A)-\eta-\Leb_{*}(\varphi(A)),
\nonumber
\end{equation}
where the last inequality follows from (\ref{est:Leb-K}) and from the fact that the compact set $g(K)$ is a competitor in the definition of $\Leb_{*}(\varphi(A))$.

Since $\eta>0$ is arbitrary, this completes the proof of (\ref{th:loc-vs-glob}).
\qed



\setcounter{equation}{0}
\section{Proof of Theorem~\ref{thm:AD} and Corollaries}\label{sec:proof-AD}

This section is organized as follows. First of all, in subsection~\ref{sec:sectioning} we apply an integral-geometric technique in order to reduce ourselves to dimension one, namely to consider functions $u$ that are defined in some interval $(a,b)\subseteq\re$ and admit an approximate derivative  at almost every $x\in(a,b)$. In this case, exactly one of the following two possibilities happens:
\begin{itemize}

\item  either there exists a set of positive measure where $\ap u'(x)\neq 0$,

\item  or $\ap u'(x)=0$ for almost every $x\in(a,b)$.

\end{itemize}

We deal with the first possibility in subsection~\ref{sec:proof-lip}, where we show that in this case it turns out that $\Fom(u,(a,b))=+\infty$ thanks to assumption (\ref{hp:omega>0}). We deal with the second possibility in subsection~\ref{sec:proof-BV}, where we show that, if $u$ is a non-constant function, then $\Fom(u,(a,b))=+\infty$ thanks to assumption (\ref{hp:omega-int}). This completes the proof of Theorem~\ref{thm:AD}.

Finally, in subsection~\ref{sec:cor} we deduce the three corollaries stated in the introduction.


\subsection{Reduction to dimension one}\label{sec:sectioning}

Since ``being constant'' is a local property and $\Omega$ is connected, it is enough to prove that $u(x)$ is essentially constant in every ball contained in $\Omega$. Therefore, up to translations and homotheties, we can assume that $\Omega$ is $B_{d}(0,1)$, namely the unit ball with center in the origin of $\re^{d}$.

Now we introduce the family of one-dimensional sections of $u$. Let $S^{d-1}$ denote the unit sphere with center in the origin of $\re^{d}$. For every $v\in S^{d-1}$, let $P_{v}$ denote the projection of $B_{d}(0,1)$ onto the hyperplane passing through the origin and orthogonal to $v$. For every $z\in P_{v}$, let us set $L_{z}:=(1-\|z\|^{2})^{1/2}$, and let us consider the function $u_{v,z}:(-L_{z},L_{z})\to\re$ defined by
\begin{equation}
u_{v,z}(t):=u(z+tv)
\qquad
\forall t\in (-L_{z},L_{z}).
\label{defn:1D-sections}
\end{equation}

In the sequel we need the following key properties of these sections.

\begin{itemize}

\item  (Integral-geometric representation of the functional). For every measurable function $g:B_{d}(0,1)^{2}\to[0,+\infty)$, a change of variables in the integral shows that
\begin{equation}
\iint_{B_{d}(0,1)^{2}}g(x,y)\,dx\,dy=
\frac{1}{2}\int_{S^{d-1}}dv\int_{P_{v}}dz\int_{(-L_{z},L_{z})^{2}}g(z+vt,z+vs)|t-s|^{d-1}\,dt\,ds,
\nonumber
\end{equation}
and in particular
\begin{equation}
\Fom(u,B_{d}(0,1))=\frac{1}{2}\int_{S^{d-1}}dv\int_{P_{v}}\Fom(u_{v,z},(-L_{z},L_{z}))\,dz.
\label{eqn:int-geom}
\end{equation}

In other words, the functional (\ref{defn:F-omega}) computed in a function $u$ of $d$ variables is some sort of average of the same functional computed on the one-dimensional sections $u_{v,z}$ of $u$.

\item  (Approximate differentiability of sections). If $u$ is approximately differentiable for almost every $x\in B_{d}(0,1)$, then for every $v\in S^{d-1}$ it turns out that, for almost every $z\in P_{v}$, the one-dimensional section $u_{v,z}(t)$ is approximately differentiable at $t$ for almost every $t$ in $(-L_{z},L_{z})$. This fact can be deduced in a rather standard way from the characterization of approximately differentiable functions as those functions that coincide with functions of class $C^{1}$ up to a set with arbitrarily small measure (see~\cite[Theorem~3.1.16]{federer:GMT}).

\item  (Characterization of constant functions through their sections). If for almost every $v\in S^{d-1}$ it turns out that for almost every $z\in P_{v}$ the one-dimensional section $u_{v,z}(t)$ is essentially constant in the interval $(-L_{z},L_{z})$, then $u(x)$ is essentially constant in $B_{d}(0,1)$. This property can be proved as an exercise in measure theory, or even deduced from (\ref{eqn:int-geom}) in the following way: if ``almost every'' section is constant, then the right-hand side is zero, and hence also the left-hand side is zero, and this is possible only if $u$ is constant. An analogous problem is presented in the final Lemma of~\cite{2002-brezis} and in~\cite[Lemma~2]{1999-TMNA-BreLiMirNir}.

\end{itemize}


\subsection{When the approximate derivative is not identically zero}\label{sec:proof-lip}

In this section we show that $\Fom(u,(a,b))=+\infty$ provided that the positivity assumption (\ref{hp:omega>0}) is satisfied, and the approximate derivative of $u$ exists, and is different from~0, on a subset of $(a,b)$ with positive measure. To be more precise, we prove actually a stronger result, where we assume just that the difference quotients of $u$ remain ``often enough'' uniformly bounded away from~0 and infinity on a set with positive measure (and this happens if $u$ is approximately differentiable with $\ap u'(x)\neq 0$ on a set with positive measure).

\begin{thm}\label{thm:u'>0}

Let $\omega:[0,+\infty)\to[0,+\infty)$ be a continuous function that satisfies the positivity assumption (\ref{hp:omega>0}). Let $u:(a,b)\to\re$ be a measurable function, and let $Ru(x,y)$ denote its difference quotients defined in analogy with (\ref{defn:dq}).

Let us assume that there exist two real numbers $0<\mu_{1}<\mu_{2}$, and a measurable set $A\subseteq(a,b)$, such that
\begin{enumerate}
\renewcommand{\labelenumi}{(\roman{enumi})}

\item  $\Leb(A)>0$,

\item  for every $x\in A$, the set
\begin{equation}
E(x):=\left\{h\in(0,b-x):\mu_{1}\leq|Ru(x,x+h)|\leq\mu_{2}\right\}
\nonumber
\end{equation}
has positive upper density in~$0$, namely
\begin{equation}
\limsup_{r\to 0^{+}}\frac{\Leb\left(E(x)\cap(0,r)\strut\right)}{r}
>0
\qquad
\forall x\in A.
\nonumber
\end{equation}

\end{enumerate} 

Then it turns out that $\Fom(u,(a,b))=+\infty$.

\end{thm}

\begin{proof}

Let us set
\begin{equation}
\omega_{0}:=\min\left\{\omega(\sigma):\mu_{1}\leq\sigma\leq\mu_{2}\right\},
\nonumber
\end{equation}
and let us observe that $\omega_{0}>0$ because of assumption (\ref{hp:omega>0}). Now with the change of variable $h=y-x$ we obtain that
\begin{equation}
\Fom(u,(a,b))\geq
\int_{A}dx\int_{E(x)}\frac{\omega(|Ru(x,x+h)|)}{h}\,dh\geq
\omega_{0}\int_{A}dx\int_{E(x)}\frac{1}{h}\,dh.
\nonumber
\end{equation}

Since $E(x)$ has positive upper density in~0, from Lemma~\ref{lemma:pos-dens} we deduce that the integral over $E(x)$ diverges for every $x\in A$. Since $A$ has positive Lebesgue measure, the double integral diverges.
\end{proof}



\subsection{When the approximate derivative vanishes identically}\label{sec:proof-BV}

In this section we show that $\Fom(u,(a,b))=+\infty$ provided that $\omega$ satisfies (\ref{hp:omega-int}), the function $u$ is non-constant, and its approximate derivative vanishes almost everywhere in $(a,b)$. To this end, we consider the difference quotients $Ru(x,y)$ defined in analogy with (\ref{defn:dq}), and the function $\gamma(\mu)$ defined by (\ref{defn:gamma(mu)}). The key point is establishing inequality (\ref{th:gamma>omega}) for suitable positive values of $c_{0}$ and $\mu_{0}$, because then we can apply Lemma~\ref{lemma:disintegration} to the measure $\nu$ on $\re$ defined by 
\begin{equation}
\nu(M):=\iint_{Z(u,(a,b),M)}\frac{\omega(|Ru(x,y)|)}{y-x}\,dx\,dy
\nonumber
\end{equation}
for every Borel subset $M\subseteq\re$, and deduce that
\begin{equation}
\Fom(u,(a,b))\geq
\nu([\mu_{0},+\infty))\geq
c_{0}\int_{\mu_{0}}^{+\infty}\frac{\omega(\mu)}{\mu^{2}}\,d\mu,
\nonumber
\end{equation}
from which the conclusion follows because of assumption (\ref{hp:omega-int}). Moreover, since
\begin{equation}
\iint_{Z(u,(a,b),[\mu-\delta,\mu])}\frac{\omega(|Ru(x,y)|)}{y-x}\,dx\,dy\geq
\min_{\sigma\in[\mu-\delta,\mu]}\omega(\sigma)\cdot\iint_{Z(u,(a,b),[\mu-\delta,\mu])}\frac{1}{y-x}\,dx\,dy,
\nonumber
\end{equation}
the proof of (\ref{th:gamma>omega}) is equivalent to showing that
\begin{equation}
\liminf_{\delta\to 0^{+}}\frac{1}{\delta}
\iint_{Z(u,(a,b),[\mu-\delta,\mu])}\frac{1}{y-x}\,dx\,dy\geq
\frac{c_{0}}{\mu^{2}}
\qquad
\forall\mu\geq\mu_{0}.
\label{th:c0-mu0}
\end{equation}
and this is the point where Theorem~\ref{thm:loc-vs-glob} comes into play.

Let us introduce some notation. Since the function $u(x)$ is non-constant, there exist two Lebesgue points where it has distinct values. Up to restricting the interval, we can always assume that the two points are the endpoints of the interval. In addition, up to adding a suitable constant and/or changing the sign, we can also assume that $u$ is negative in $a$ and positive in $b$, and more precisely that there exists a positive real number $J$ such that
\begin{equation}
\aplimsup_{x\to a^{+}}u(x)<-J<J<\apliminf_{x\to b^{-}}u(x).
\label{hp:J}
\end{equation}

Let us set $M:=\|u\|_{L^{\infty}((a,b))}$, and let us choose $\eta_{0}\in(0,(b-a)/2)$ such that the set
\begin{equation}
\mathcal{S}(\eta):=\{x\in(a,a+\eta):u(x)> -J\}\cup\{x\in(b-\eta,b):u(x)< J\}
\nonumber
\end{equation}
satisfies
\begin{equation}
\Leb(\mathcal{S}(\eta))\leq\frac{J}{M+J}\,\eta
\qquad
\forall\eta\in(0,\eta_{0}).
\label{defn:eta-0}
\end{equation}

We point out that $\eta_{0}$ exists because of (\ref{hp:J}). We claim that (\ref{th:c0-mu0}) holds true with
\begin{equation}
c_{0}:=J,
\qquad\qquad
\mu_{0}:=\frac{M+J}{\eta_{0}}.
\nonumber
\end{equation}

To this end, for every $\mu>\mu_{0}$ we consider the set $\mathcal{S}(\eta)$ corresponding to 
\begin{equation}
\eta:=\frac{M+J}{\mu},
\label{defn:eta}
\end{equation}
then we introduce the set
\begin{equation}
A_{\mu}:=(a,b)\setminus\mathcal{S}(\eta)
\nonumber
\end{equation}
and the function $\varphi_{\mu}:(a,b)\to\re$ defined by
\begin{equation}
\varphi_{\mu}(x):=x-\frac{u(x)}{\mu}
\qquad
\forall x\in(a,b).
\label{defn:phi-mu}
\end{equation}

Since $\ap u'(x)=0$ for almost every $x\in(a,b)$, the function $\varphi_{\mu}(x)$ is approximately differentiable at  almost every $x\in(a,b)$, and
\begin{equation}
\ap \varphi_{\mu}'(x)=1
\qquad
\text{for almost every }x\in(a,b),
\nonumber
\end{equation}
so that in particular $\varphi_{\mu}$ satisfies the assumption of Theorem~\ref{thm:loc-vs-glob}, both with domain $(a,b)$ and with domain $A_{\mu}$. We observe also that
\begin{equation}
Z(u,(a,b),[\mu-\delta,\mu])=
Z(\varphi_{\mu},(a,b),[0,\delta/\mu])\supseteq
Z(\varphi_{\mu},A_{\mu},[0,\delta/\mu]),
\nonumber
\end{equation}
and therefore from Theorem~\ref{thm:loc-vs-glob} we deduce that
\begin{eqnarray}
\liminf_{\delta\to 0^{+}}\frac{1}{\delta}
\iint_{Z(u,(a,b),[\mu-\delta,\mu])}\frac{1}{y-x}\,dx\,dy & \geq & 
\liminf_{\delta\to 0^{+}}\frac{1}{\mu}\frac{\mu}{\delta}
\iint_{Z(\varphi_{\mu},A_{\mu},[0,\delta/\mu])}\frac{1}{y-x}\,dx\,dy
\nonumber
\\[1ex]
& \geq &
\frac{1}{\mu}\left(\Leb(A_{\mu})-\Leb_{*}(\varphi_{\mu}(A_{\mu}))\strut\right).
\label{est:Leb-Leb}
\end{eqnarray}

It remains to estimate $\Leb(A_{\mu})$ and $\Leb_{*}(\varphi_{\mu}(A_{\mu}))$.

Since $\mu>\mu_{0}$, we obtain that $\eta<\eta_{0}$ and therefore from (\ref{defn:eta-0}) and (\ref{defn:eta}) we deduce that $\Leb(\mathcal{S}(\eta))\leq J/\mu$, and therefore
\begin{equation}
\Leb(A_{\mu})=(b-a)-\Leb(\mathcal{S}(\eta))\geq (b-a)-\frac{J}{\mu}.
\label{est:Leb(A)}
\end{equation}

Now we claim that
\begin{equation}
a+\frac{J}{\mu}\leq\varphi_{\mu}(x)\leq b-\frac{J}{\mu}
\qquad
\forall x\in A_{\mu}.
\label{est:phi-mu}
\end{equation}

If we prove this claim, then we deduce that
\begin{equation}
\Leb_{*}(\varphi_{\mu}(A_{\mu}))\leq (b-a)-\frac{2J}{\mu},
\label{est:Leb(phi(A))}
\end{equation}
and plugging (\ref{est:Leb(A)}) and (\ref{est:Leb(phi(A))}) into (\ref{est:Leb-Leb}) we obtain (\ref{th:c0-mu0}) with $c_{0}=J$.

So it remains to prove (\ref{est:phi-mu}). Let us consider the estimate from above (the other one is symmetric). If $x\in(a,b-\eta]$, then from (\ref{defn:eta}) we obtain that
\begin{equation}
\varphi_{\mu}(x)\leq
b-\eta+\frac{M}{\mu}=
b-\frac{M+J}{\mu}+\frac{M}{\mu}=
b-\frac{J}{\mu}.
\nonumber
\end{equation}

If $x\in(b-\eta,b)\cap A_{\mu}$, then $u(x)\geq J$ (because we removed the points of $\mathcal{S}(\eta)$), and therefore
\begin{equation}
\varphi_{\mu}(x)\leq 
b-\frac{u(x)}{\mu}\leq
b-\frac{J}{\mu}.
\nonumber
\end{equation}

This completes the proof.
\hfill
$\Box$

\begin{rmk}\label{rmk:sigma-eta}
\begin{em}

The definition of $A_{\mu}$ and the estimate of its Lebesgue measure is the technical point where we need the assumption that $u\in L^{\infty}((a,b))$. If we assume that $u$ is continuous at the endpoints, or at least that $\mathcal{S}(\eta_{0})$ is empty for some fixed $\eta_{0}>0$, then we can relax the assumption to $u\in L^{1}((a,b))$. Indeed, in this case we can define
\begin{equation}
A_{\mu}:=\{x\in(a,b):|u(x)|\leq\eta_{0}\mu-J\},
\nonumber
\end{equation}
and obtain again that (\ref{est:phi-mu}) holds true, so that again also (\ref{est:Leb(phi(A))}) holds true. In addition, if $u\in L^{1}((a,b))$ we know that
\begin{equation}
\Leb(A_{\mu})=(b-a)-o\left(\frac{1}{\mu}\right)
\nonumber
\end{equation}
as $\mu\to +\infty$, and hence also (\ref{est:Leb(A)}) is satisfied when $\mu$ is large enough. This is enough to apply Theorem~\ref{thm:loc-vs-glob} in order to deduce (\ref{th:c0-mu0}).

Roughly speaking, the definition of $A_{\mu}$ has to satisfy two opposing needs. On the one hand, we need (\ref{est:phi-mu}) to be true, and this forces us to remove from $(a,b)$ the points where $|u(x)|$ is ``too large'', and the points near the boundary where the values of $u(x)$ are not close enough to the expected limits. On the other hand, we need (\ref{est:Leb(A)}) to be true, which forces us to remove as few points as possible. Finding the correct balance between the two needs seems to be a challenging problem (see Open problem~\ref{open:2L1}).

\end{em}
\end{rmk}


\subsection{Proof of Corollaries}\label{sec:cor}

\paragraph{Proof of Corollary~\ref{cor:BV}}

As in the proof of Theorem~\ref{thm:AD} we reduce ourselves to the case where $\Omega$ is a ball in $\re^{d}$, and then we consider the one-dimensional sections $u_{v,z}(t)$ defined in (\ref{defn:1D-sections}). If $u$ is a function with bounded variation, it is well-known that, for every $v\in S^{d-1}$, it turns out that for almost every $z\in P_{v}$ the one-dimensional section $u_{v,z}(t)$ is a bounded variation function of one variable, and in particular it is essentially bounded (with a bound that depends on $v$ and $z$) and differentiable (not only in the approximate sense) almost everywhere. Thus from Theorem~\ref{thm:AD} we deduce that almost all sections are essentially constant, and hence also $u$ is essentially constant.
\qed

\paragraph{Proof of Corollary~\ref{cor:limit}}

Again we reduce ourselves to prove the result in dimension one for functions defined in some interval $(a,b)\subseteq\re$.

We observe that assumption (\ref{hp:omega-liminf>0}) implies in particular that
\begin{equation}
\inf\{\omega(\mu):\mu\geq\mu_{1}\}>0
\qquad
\forall\mu_{1}>0,
\nonumber
\end{equation}
and this is enough to extend Theorem~\ref{thm:u'>0} (with an analogous proof) to the case where $\mu_{2}=+\infty$. We deduce that, if $\Fom(u,(a,b))<+\infty$, then necessarily $\ap u'(x)=0$ for almost every $x\in(a,b)$, and we conclude thanks to Theorem~\ref{thm:AD}.
\qed

\paragraph{Proof of Corollary~\ref{cor:monotone}}

The monotonicity of $\omega$ implies that a truncation of $u$ does not increase $\Fom(u,\Omega)$ and therefore, if there exists a non-constant function $u$ such that $\Fom(u,\Omega)$ is finite, then there exists also a non-constant \emph{bounded} function $u$ that still satisfies $\Fom(u,\Omega)<+\infty$. On the other hand, the monotonicity of $\omega$ implies also that $\omega$ satisfies (\ref{hp:omega-liminf>0}). At this point the conclusion follows from Corollary~\ref{cor:limit}. 
\qed



\setcounter{equation}{0}
\section{Open problems and further perspectives}\label{sec:further}

In this final section we collect some of the problems that remain open. For the sake of simplicity, we state them for functions of one real variable, but we recall that the argument of section~\ref{sec:sectioning} shows that the problem is essentially one-dimensional.

Now we know that the answer to Question~\ref{open:ignat} is negative in its full generality, but positive under additional assumptions on $u$ and/or $\omega$ (see Table~\ref{table:SoA}). In both cases the threshold between positive/negative answers remains unclear. Concerning the assumption on $\omega$, in our counterexample $\omega(\mu)$ grows at infinity more that any power of $\log\mu$, but less than every power of $\mu$. On the other hand, we know from~\cite[Theorem~1.3]{2005-ignat} that a linear growth of $\omega(\mu)$ is enough for the validity of implication (\ref{implication}). So it could be interesting to investigate the gap between logarithmic and linear growth, provided of course that we keep the necessary integral condition (\ref{hp:omega-int}).

\begin{newopen}[Growth conditions on $\omega$]

Let $\omega:[0,+\infty)\to[0,+\infty)$ be a continuous function that satisfies (\ref{hp:omega>0}) and (\ref{hp:omega-int}), and in addition
\begin{equation}
\liminf_{\mu\to 0^{+}}\frac{\omega(\mu)}{\mu^{\theta}}>0
\nonumber
\end{equation}
for some $\theta\in(0,1)$. Can we conclude that $\Fom(u,(a,b))=+\infty$ for every non-constant measurable function $u:(a,b)\to \re$? 

\end{newopen}

Concerning the assumptions on $u$, we observe that the function $u$ in our counterexample is not very summable, in the sense that
\begin{equation}
\int_{0}^{1}u(x)^{\alpha}\,dx=+\infty
\qquad
\forall\alpha>0.
\label{hp:int-u}
\end{equation}

So it is reasonable to ask whether some kind of summability of $u$ can guarantee the validity of implication (\ref{implication}).

\begin{newopen}[Summability assumptions on $u$]\label{open:summability}

Let $\omega:[0,+\infty)\to[0,+\infty)$ be a continuous function that satisfies (\ref{hp:omega>0}) and (\ref{hp:omega-int}). Let us assume that the integral in (\ref{hp:int-u}) converges for some $\alpha>0$.

If $u$ is non-constant, can we deduce that $\Fom(u,(a,b))=+\infty$? 

\end{newopen}

Of course one could consider also summability conditions on $u$ that are more general than convergence of integrals, and stated for example in terms of decay properties, as $M\to +\infty$, of the measure of the set where $|u(x)|\geq M$. As far as we know, it is also conceivable that the ``optimal theory'' involves some sort of compensation between the growth of $\omega$ and the summability of $u$.

Our positive result in Theorem~\ref{thm:AD} relies on the assumption that $u$ is both bounded and approximately differentiable. Therefore, it could be interesting to investigate the intermediate stages between Theorem~\ref{thm:AD} and Open problem~\ref{open:summability} where we consider only one of the two assumptions. Concerning boundedness, we have already observed in Remark~\ref{rmk:sigma-eta} that this assumption comes into play just in a technical point. This leads us to suspect that it can be weakened (but not too much, because of our counterexample), as in the following question. 

\begin{newopen}[Unbounded functions that admit an approximate derivative]\label{open:2L1}

Let $\omega:[0,+\infty)\to[0,+\infty)$ be a continuous function that satisfies (\ref{hp:omega>0}) and (\ref{hp:omega-int}). Let $u:(a,b)\to\re$ be a non-constant function such that
\begin{itemize}

\item  its approximate derivative exists and vanishes for almost every $x\in(a,b)$,

\item  the integral (\ref{hp:int-u}) converges for some $\alpha>0$.

\end{itemize}

Can we conclude that that $\Fom(u,(a,b))=+\infty$?
\end{newopen}

We recall that the approximate differentiability of $u$ follows from the finiteness of the functional when $\omega$ satisfies (\ref{hp:omega-liminf}) (see the proof of Corollary~\ref{cor:limit}), and therefore Open problem~\ref{open:2L1} generalizes the rightmost question marks in Table~\ref{table:SoA}.

The other possible intermediate stage between Theorem~\ref{thm:AD} and Open problem~\ref{open:summability} is the case where we keep the boundedness assumption on $u$, and if needed we reinforce it by asking also continuity or H\"older continuity (as already suggested in~\cite{2005-ignat}), but we remove the assumption of approximate differentiability. A possible approach to this case could be extending our Theorem~\ref{thm:loc-vs-glob} to less regular functions. Indeed, we suspect that (\ref{th:loc-vs-glob}) could be true even if in assumption (\ref{hp:loc-exp}) we replace the approximate derivative with difference quotients.

\begin{newopen}[Local expansion vs global contraction revisited]\label{open:loc-vs-glob}

Let $A\subseteq\re$ be a bounded measurable set, and let $\varphi:A\to\re$ be a measurable function such that
\begin{equation}
\ap\liminf_{h\to 0^{+}}\left|\frac{\varphi(x+h)-\varphi(x)}{h}\right|\geq 1
\qquad
\text{for almost every }x\in A.
\label{hp:loc-exp-open}
\end{equation}

Can we conclude that inequality (\ref{th:loc-vs-glob}) still holds true?

\end{newopen}

\begin{rmk}
\begin{em}

A positive answer to Open problem~\ref{open:loc-vs-glob} would imply a positive answer to Question~\ref{open:ignat} under the sole assumption that $u$ is in $L^{\infty}$, without any differentiability requirement. The proof would be analogous to the one presented here in section~\ref{sec:proof-AD}. Indeed, either $u$ satisfies the assumptions of Theorem~\ref{thm:u'>0}, or the function $\varphi_{\mu}$ defined by (\ref{defn:phi-mu}) satisfies (\ref{hp:loc-exp-open}) for every $\mu>0$: if this is enough to deduce (\ref{th:loc-vs-glob}), then we can conclude as in section~\ref{sec:proof-BV}.

We note also that it would be enough to obtain a weaker version of (\ref{th:loc-vs-glob}), where in the left-hand side the liminf is replaced by the limsup, and the right-hand side is multiplied by a positive constant $c_{0}$, provided that $c_{0}$ is independent of $\varphi$ and $A$.

\end{em}
\end{rmk}

\begin{rmk}
\begin{em}

On the contrary, the existence of some exotic example that provides a negative answer to Open problem~\ref{open:loc-vs-glob} would not imply immediately a negative answer to Question~\ref{open:ignat} in the $L^{\infty}$ case. Indeed, any counterexample $u\in L^{\infty}((a,b))$ to Question~\ref{open:ignat} has necessarily the property that the function $\varphi_{\mu}$ defined by (\ref{defn:phi-mu}) provides a negative answer to Open problem~\ref{open:loc-vs-glob} for every $\mu$ in some subset $M\subseteq(0,+\infty)$ that is large enough so that
\begin{equation}
\int_{(0,+\infty)\setminus M}\frac{\omega(\mu)}{\mu^{2}}\,d\mu<+\infty.
\nonumber
\end{equation}

In other words, an exotic isolated counterexample to Open problem~\ref{open:loc-vs-glob} is not enough, but we need a whole family of counterexamples with a very special structure.
\end{em}
\end{rmk}

The heuristic idea behind Theorem~\ref{thm:loc-vs-glob} is that the left-hand side of (\ref{th:loc-vs-glob}) represents some sort of quantification of the ``lack of injectivity'' of $\varphi$, and the intuition suggests that it is impossible to map a set into a smaller set through a map $\varphi$ that has the local expansion property (\ref{hp:loc-exp}) without violating injectivity. The formalization of this ideas lies in the properties of the cumulative distribution, which in turn require to understand how the function $\varphi$ transforms the measure. If we assume some form of differentiability of $\varphi$, this can be done by applying some version of the area formula, but when assumptions concern only the difference quotients the situation is less clear. Just to give an extreme example, for the time being we are not able to answer the following apparently simpler question without assuming some approximate differentiability.

\begin{newopen}[Can an injective ``local expansion'' shrink the measure?]

Determine whether there exists a measurable function $\varphi:(0,2)\to(0,1)$ that is injective and satisfies
\begin{equation}
\ap\liminf_{h\to 0^{+}}\left|\frac{\varphi(x+h)-\varphi(x)}{h}\right|\geq 1
\qquad
\text{for almost every }x\in(0,2).
\label{loc-exp-open}
\end{equation}

\end{newopen}

We suspect that the answer to the previous question is negative, but we have no proof even if we reinforce (\ref{loc-exp-open}) by asking that the approximate limit is $+\infty$.

\bigskip



\subsubsection*{\centering Acknowledgments}

The first author was introduced to this problem by some personal discussions with H.~Brezis and R.~Ignat that took place during and after the workshop ``New Perspectives in Nonlinear PDE'' held in Haifa (Israel) in June 2019. For this reason, we are both grateful to them and to the organizers of the workshop. We would like to thank also L.~Ambrosio for suggesting to investigate the monotone case that led to Corollary~\ref{cor:monotone}.

The first author is a member of the \selectlanguage{italian} ``Gruppo Nazionale per l'Analisi Matematica, la Probabilità e le loro Applicazioni'' (GNAMPA) of the ``Istituto Nazionale di Alta Matematica'' (INdAM). 

\selectlanguage{english}



\label{NumeroPagine}

\end{document}